\documentclass[11pt,reqno]{amsart}

\usepackage{pgfplots}
\pgfplotsset{compat=newest}
\usetikzlibrary{shapes.geometric,arrows,fit,matrix,positioning}
\tikzset
{
    treenode/.style = {circle, draw=black, align=center, minimum size=1cm},
    subtree/.style  = {isosceles triangle, draw=black, align=center, minimum height=0.5cm, minimum width=1cm, shape border rotate=90, anchor=north}
}
\usepackage{amssymb}
\usepackage{amsmath,amsthm,amsfonts,amssymb,latexsym,mathrsfs,color,hyperref}
\usepackage{graphicx}
\usepackage{xspace}
\usepackage{multirow}
\usepackage{diagbox}
\usepackage{capt-of}
\usepackage{tikz}
\tikzset{
	level 1/.style = {sibling distance = 1.5cm},
	level 2/.style = {sibling distance = 0.8cm},
    level distance = 0.9 cm
}
\usetikzlibrary {decorations.pathmorphing}
\usetikzlibrary{matrix, positioning}
\tikzstyle{snakeline} = [decorate, decoration={snake, amplitude=.4mm, segment length=2mm}]

\usepackage{tikz-cd}
\usepackage[boxsize=1.3em, centerboxes]{ytableau}

\usepackage{tikz-qtree}
\tikzset{every tree node/.style={minimum width=0.1cm,draw,circle},
         blank/.style={draw=none},
         edge from parent/.style=
         {draw,edge from parent path={(\tikzparentnode) -- (\tikzchildnode)}},
         level distance=0.8cm}

\newtheorem{theorem}{Theorem}%[section]
\newtheorem{corollary}[theorem]{Corollary}
\newtheorem{proposition}[theorem]{Proposition}

\newtheorem{lemma}[theorem]{Lemma}
\newtheorem{definition}[theorem]{Definition}
\newtheorem{example}[theorem]{Example}

\newcommand{\ssuc}{{\rm simsuc\,}}
\newcommand{\plrm}{{\rm plrmin\,}}
\newcommand{\Fix}{{\rm Fix\,}}
\newcommand{\Aexc}{{\rm Drop\,}}
\newcommand{\md}{\mathcal{D}}

\newcommand{\val}{{\rm val}}

\newcommand{\suc}{{\rm suc\,}}
\newcommand{\basc}{{\rm basc\,}}

\newcommand{\lrm}{{\rm lrmin\,}}

\newcommand{\cda}{{\rm cda\,}}

\newcommand{\drop}{{\rm drop\,}}

\newcommand{\dasc}{{\rm dasc\,}}

\newcommand{\pk}{{\rm pk\,}}
\newcommand{\ddes}{{\rm ddes\,}}
\newcommand{\des}{{\rm des\,}}

\newcommand{\Exc}{{\rm Exc\,}}
\newcommand{\exc}{{\rm exc\,}}

\newcommand{\aexc}{{\rm drop\,}}

\newcommand{\cyc}{{\rm cyc\,}}

\newcommand{\fix}{{\rm fix\,}}

\newcommand{\mdn}{\mathcal{D}}

\newcommand{\msn}{\mathfrak{S}_n}

\newcommand{\rss}{\mathcal{SS}}
\newcommand{\ms}{\mathfrak{S}}

\newcommand{\rs}{\mathcal{RS}}

\newcommand{\lrf}[1]{\lfloor #1\rfloor}

\newcommand{\mq}{\mathcal{Q}}
\newcommand{\mqn}{\mathcal{Q}_n}

\newcommand{\asc}{{\rm asc\,}}

\newcommand{\stirling}[2]{\genfrac{[}{]}{0pt}{}{#1}{#2}}

\title{On the joint distributions of succession and Eulerian statistics}
\author[S.-M.~Ma]{Shi-Mei Ma}
\address{School of Mathematics and Statistics,
        Northeastern University at Qinhuangdao,
         Hebei 066000, P.R. China}
\email{shimeimapapers@163.com (S.-M. Ma)}
\author[H.~Qi]{Hao Qi}
\address{College of Mathematics and Physics, Wenzhou University, Wenzhou 325035, P.R. China}
\email{qihao@wzu.edu.cn(H.~Qi)}
\author[J.~Yeh]{Jean Yeh}
\address{Department of Mathematics, National Kaohsiung Normal University, Kaohsiung 82444, Taiwan}
\email{chunchenyeh@nknu.edu.tw}
\author[Y.-N. Yeh]{Yeong-Nan Yeh}
\address{College of Mathematics and Physics, Wenzhou University, Wenzhou 325035, P.R. China}
\email{mayeh@math.sinica.edu.tw (Y.-N. Yeh)}
\begin{document}

\maketitle
\begin{abstract}
The motivation of this paper is to investigate the joint distribution of succession and Eulerian statistics.
We first investigate the enumerators for the joint distribution of descents, big ascents and successions over all permutations in the symmetric group.
As an generalization a result of Diaconis-Evans-Graham (Adv. in Appl. Math., 61 (2014), 102--124),
we show that two triple set-valued statistics of permutations are equidistributed on symmetric groups.
We then introduce the definition of proper left-to-right minimum, and discover that the joint distribution
of the succession and proper left-to-right minimum statistics over permutations is a symmetric distribution.
In the final part, we discuss the relationship between the $\fix$ and $\cyc$ $(p,q)$-Eulerian polynomials and
the joint distribution of succession and Eulerian-type statistics. In particular, we give a concise derivation of the generating function for a
six-variable Eulerian polynomials.
\bigskip

\end{abstract}
\date{\today}

\tableofcontents
%\date{\today}
%%%%%%%%%%%%%%%%%%%%%%%%%%%%%%%%%%%%%%%%%%%%%%%%%%%%%%%%%%%%%%%%%%%%%%%%%%
%%%%%%%%%%%%%%%%%%%%%%%%%%%%%%%%%%%%%%%%%%%%%%%%%%%%%%%%%%%%%%%%%%%%%%%%%%
%%%%%%%%%%%%%%%%%%%%%%%%%%%%%%%%%%%%%%%%%%%%%%%%%%%%%%%%%%%%%%%%%%%%%%%%%%
%%%%%%%%%%%%%%%%%%%%%%%%%%%%%%%%%%%%%%%%%%%%%%%%%%%%%%%%%%%%%%%%%%%%%%%%%%
%%%%%%%%%%%%%%%%%%%%%%%%%%%%%%%%%%%%%%%%%%%%%%%%%%%%%%%%%%%%%%%%%%%%%%%%%%
%%%%%%%%%%%%%%%%%%%%%%%%%%%%%%%%%%%%%%%%%%%%%%%%%%%%%%%%%%%%%%%%%%%%%%%%%%
%%%%%%%%%%%%%%%%%%%%%%%%%%%%%%%%%%%%%%%%%%%
\section{Introduction}
%%%%%%%%%%%%%%%%%%%%%%%%%%%%%%%%%%%%%%%%%%%
%%%%%%%%%%%%%%%%%%%%%%%%%%%%%%%%%%%%%%%%%%%%%%%%%%%%%%%%%%%%%%%%%%%%%%%%%%%%%%%%%
%%%%%%%%%%%%%%%%%%%%%%%%%%%%%%%%%%%%%%%%%%%%%%%%%%%%%%%%%%%%%%%%%%%%%%%%%%%%%%%%%
%%%%%%%%%%%%%%%%%%%%%%%%%%%%%%%%%%%%%%%%%%%%%%%%%%
Let $\msn$ denote the symmetric group of all permutations of $[n]$, where $[n]=\{1,2,\ldots,n\}$.
As usual, we write $\pi=\pi(1)\pi(2)\cdots\pi(n)\in\msn$.
A {\it fixed point} of $\pi\in\msn$ is an index $k\in [n]$ such that $\pi(k)=k$. Let $\fix(\pi)$ be the number of fixed points of $\pi$.
We say that $\pi$ is a {\it derangement} if it has no fixed points.
Denote by $\mdn_n$ the set of all derangements in $\msn$, and let $d_n=\#\mdn_n$ be the {\it derangement number}. It is well known that
\begin{equation}\label{dn-explicit}
d_n=n!\sum_{i=0}^n\frac{(-1)^i}{i!}.
\end{equation}
Derangements have been studied from various perspectives, see~\cite{Wachs93,Zhang22} for surveys on this topic. For example,
D\'{e}sarm\'{e}nien-Wachs~\cite{Wachs93} constructed a bijection between descent classes of derangements and descent classes of
desarrangements (a desarrangement is a permutation whose first ascent is even). Recently,
Gustafsson-Solus~\cite{Gustafsson20} investigated the geometric interpretation of derangement polynomials.

The enumeration of finite sequences according to the number of successions was initiated by Kaplansky and Riordan in the 1940s~\cite{Kaplansky,Riordan45}.
There are several variants of successions and they have been extensively studied on various structures, including permutations, 
increasing trees, set partitions, compositions and integer partitions, see~\cite{Brenti18,Diaconis14,Dymacek,Mansour16,Tanny} for instances.
A {\it succession} of $\pi\in\msn$ is an index $k\in [n-1]$ such that $\pi(k+1)=\pi(k)+1$, and $\pi(k)$ is called a {\it succession value}.
Let $\suc(\pi)$ be the number of successions of $\pi$.
The joint distribution of ascents and successions over permutations has been explored by Roselle~\cite{Roselle68} and Dymacek-Roselle~\cite{Dymacek}.
Let $q_n=\#\{\pi\in\msn:~\suc(\pi)=0\}$.
According to~\cite[Eq~(3.8)]{Roselle68}, one has
\begin{equation}\label{qndn}
q_n=d_n+d_{n-1}.
\end{equation}

Following~\cite{Brualdi92}, a {\it relative derangement} on $[n]$ is a permutation in $\msn$ with no successions.
Using the principle of inclusion and exclusion,
Brualdi~\cite[Theorem~6.5.1]{Brualdi92} independently found that
$$q_n=(n-1)!\sum_{i=0}^n\frac{(-1)^i(n-i)}{i!}.$$
Combining this explicit formula with~\eqref{dn-explicit},
he rediscovered the identity~\eqref{qndn}. A combinatorial interpretation of~\eqref{qndn} has been obtained by Chen~\cite{Chen96} by introducing skew derangements.

Recently, Diaconis-Evans-Graham~\cite{Diaconis14} found that
for all $I\subseteq [n-1]$, one has
\begin{equation}\label{Graham}
\begin{aligned}
&\#\{\pi\in\msn: \{k\in [n-1]: \pi(k+1)=\pi(k)+1\}= I\}\\
&=\#\{\pi\in\msn: \{k\in [n-1]: \pi(k)=k\}= I\}.
\end{aligned}
\end{equation}
They presented three different proofs of it, including an enumerative proof, a Markov chain proof and a bijective proof.
In~\cite{Brenti18}, Brenti-Marietti extended the notion of succession for ordinary permutations to
adjacent ascent of colored permutations.

The organization of this paper is as follows. In Section~\ref{section02}, we collect
the definitions and preliminary results that will be used in the sequel.
In Section~\ref{section03}, we investigate the enumerators for the joint distribution of descents, big ascents and successions over all permutations in the symmetric group.
As an generalization of~\eqref{Graham},
we show that two triple set-valued statistics of permutations are equidistributed.
Then we introduce the definition of proper left-to-right minimum.
Let $\plrm(\pi)$ and $\cyc(\pi)$ denote the numbers of proper left-to-right minima and cycles of $\pi$, respectively.
In Section~\ref{section04}, we study the relationship between the $\fix$ and $\cyc$ $(p,q)$-Eulerian polynomials and
the joint distribution of succession and Eulerian statistics.
As an illustration, a special case of Theorem~\ref{Anpq} says that
\begin{equation*}
\sum_{\pi\in\ms_{n+1}}s^{\suc(\pi)}t^{\plrm(\pi)}=\sum_{\pi\in\msn}{\left(\frac{t+s}{2}\right)}^{\fix(\pi)}2^{\cyc(\pi)},
\end{equation*}
which implies that $(\suc,~\plrm)$ is a symmetric distribution over permutations.
%%%%%%%%%%%%%%%%%%%%%%%%%%%%%%%%%%%%%%%%%%%%%%%%%%%%%%%%%%%
%%%%%%%%%%%%%%%%%%%%%%%%%%%5%%%%
%%%%%%%%%%%%%%%%%%%%%%%%%%%%%%%%%%%%%%%%%%%%%%%%%%%%%%%%%%%
%%%%%%%%%%%%%%%%%%%%%%%%%%%5%%%%
%%%%%%%%%%%%%%%%%%%%%%%%%%%%%%%%%%%%%%%%%%%%%%%%%%%%%%%%%%%
%%%%%%%%%%%%%%%%%%%%%%%%%%%5%%%%
\section{Notation and preliminary results}\label{section02}
%%%%%%%%%%%%%%%%%%%%%%%%%%%%
%\hspace*{\parindent}
%%\subsection{%%%%%%%%%%%%%%%%%%%%%%%%%%%%%%%%%%%%%%%%%%%%%%%%%%%
%%%%%%%%%%%%%%%%%%%%%%%%%%%%%%%%%%%%%%%%%%%%%%%%%%%%%%%%%%%
%%%%%%%%%%%%%%%%%%%%%%%%%%%5%%%%
There has been much work on the symmetric expansions of polynomials, see~\cite{Athanasiadis20,Lin21,Ma2401,Ma2402} and references therein.
Let $f(x)=\sum_{i=0}^nf_ix^i$ be a polynomial with real coefficients.
If $f(x)$ is symmetric, i.e., $f_i=f_{n-i}$ for all indices $0\leqslant i\leqslant n$,
then it can be expanded uniquely as
$$f(x)=\sum_{k=0}^{\lrf{{n}/{2}}}\gamma_kx^k(1+x)^{n-2k},$$
and it is said to be {\it $\gamma$-positive}
if $\gamma_k\geqslant 0$ for all $0\leqslant k\leqslant \lrf{{n}/{2}}$.
We say that the polynomial $f(x)$ is {\it spiral} if
$$f_n\leqslant f_0\leqslant f_{n-1}\leqslant f_1\leqslant \cdots\leqslant f_{\lrf{n/2}}.$$
and it is said to be {\it alternatingly increasing} if
$$f_0\leqslant f_n\leqslant f_1\leqslant f_{n-1}\leqslant\cdots \leqslant f_{\lrf{{(n+1)}/{2}}}.$$
If $f(x)$ is spiral and $\deg f(x)=n$, then $x^nf(1/x)$ is alternatingly increasing, and vice versa.
From~\cite[Remark~2.5]{Beck2010}, we see that $f(x)$ has a unique decomposition $f(x)= a(x)+xb(x)$, where
\begin{equation}\label{ax-bx-prop01}
a(x)=\frac{f(x)-x^{n+1}f(1/x)}{1-x},~b(x)=\frac{x^nf(1/x)-f(x)}{1-x}.
\end{equation}
When $f(0)\neq0$, we have $\deg a(x)=n$ and $\deg b(x)\leqslant n-1$. Note that $a(x)$ and $b(x)$ are both symmetric.
We call the ordered pair of polynomials $(a(x),b(x))$ the {\it symmetric decomposition} of $f(x)$.
Br\"and\'en-Solus~\cite{Branden21} pointed out that
$f(x)$ is alternatingly increasing if and only if the pair of polynomials in its symmetric decomposition are both unimodal
and have only nonnegative coefficients. Following~\cite[Definition~1.2]{Ma2401}, the polynomial $f(x)$ is said to be {\it bi-$\gamma$-positive} if
$a(x)$ and $b(x)$ are both $\gamma$-positive. Thus bi-$\gamma$-positivity is stronger than alternatingly increasing property, see~\cite{Athanasiadis20,Branden22,Han21,Ma2401} for the recent progress on this subject.
In this paper, we shall present several new $\gamma$-positive or bi-$\gamma$-positive polynomials.

Let $\pi\in\msn$. A {\it descent} (resp.~{\it ascent, excedance}) of $\pi$ is an index $i\in[n-1]$
such that $\pi(i)>\pi(i+1)$ (resp.~$\pi(i)<\pi(i+1)$, $\pi(i)>i$). Let $\des(\pi)$ (resp.~$\asc(\pi)$, $\exc(\pi)$) denote the number of descents (resp.~ascents, excedances) of $\pi$.
It is well known that descents, ascents and excedances are equidistributed over the symmetric groups,
and their common enumerative polynomials are the {\it Eulerian polynomials} $A_n(x)$, i.e.,
$$A_n(x)=\sum_{\pi\in\msn}x^{\des(\pi)}=\sum_{\pi\in\msn}x^{\asc(\pi)}=\sum_{\pi\in\msn}x^{\exc(\pi)}.$$
The {\it derangement polynomials} are defined by $$d_n(x)=\sum_{\pi\in\mdn_n}x^{\exc(\pi)}.$$
In the theory of subdivisions of simplicial complexes, the Eulerian polynomial $A_n(x)$ arises as the
$h$-polynomial of the barycentric subdivision of a simplex and derangement polynomial $d_n(x)$
as its local $h$-polynomial, see~\cite{Gustafsson20,Stanley92} for details.

Below are the first few Eulerian and derangement polynomials:
$$A_0(x)=A_1(x)=1,~A_2(x)=1+x,~A_3(x)=1+4x+x^2,~A_4(x)=1+11x+11x^2+x^3;$$
$$d_0(x)=1,~d_1(x)=0,~d_2(x)=x,~d_3(x)=x+x^2,~d_4(x)=x+7x^2+x^3.$$
The generating function of $d_n(x)$ is given as follows (see~\cite[Proposition~6]{Brenti90}):
\begin{equation}\label{dxz-EGF}
d(x;z)=\sum_{n=0}^\infty d_n(x)\frac{z^n}{n!}=\frac{1-x}{\mathrm{e}^{xz}-x\mathrm{e}^{z}}.
\end{equation}

We say that an index $i$ is a {\it double descent} of $\pi\in\msn$ if $\pi(i-1)>\pi(i)>\pi(i+1)$, where $\pi(0)=\pi(n+1)=0$.
Foata-Sch\"utzenberger~\cite{Foata70} discovered the following remarkable result:
\begin{equation}\label{Anx-gamma}
A_n(x)=\sum_{i=0}^{\lrf{(n-1)/2}}\gamma_{n,i}x^i(1+x)^{n-1-2i},
\end{equation}
where $\gamma_{n,i}$ is the number of permutations in $\msn$ with $i$ descents and have no double descents.
Let $\cda(\pi)=\#\{i:\pi^{-1}(i)<i<\pi(i)\}$ be the number of {\it cycle double ascents} of $\pi$.
Using the theory of continued fractions, Shin-Zeng~\cite[Theorem~11]{Zeng12} obtained that
\begin{equation}\label{dnxq-gamma}
d_n(x,q)=\sum_{\pi\in\md_n}x^{\exc(\pi)}q^{\cyc(\pi)}=\sum_{k=1}^{\lrf{n/2}}\sum_{\pi\in \md_{n,k}}q^{\cyc(\pi)}x^{k}(1+x)^{n-2k}.
\end{equation}
where $\md_{n,k}=\{\pi\in \msn: \fix(\pi)=0,~\cda(\pi)=0,~\exc(\pi)=k\}$.
When $q=1$, it follows from~\eqref{dnxq-gamma} that $d_n(x)$ is $\gamma$-positive.

Let $P(n,r,s)$ be the number of permutations in $\msn$ with $r$ ascents and $s$ successions.
Roselle~\cite[Eq.~(2.1)]{Roselle68} proved that
$$P(n,r,s)=\binom{n-1}{s}P(n-s,r-s,0).$$
Let $\mqn$ be the set of permutations in $\msn$ with no successions.
Let $P^*_n(x)=\sum_{r=1}^{n-1}P^*(n,r)x^r$, where $P^*(n,r)=\#\{\pi\in\mqn: \asc(\pi)=r-1,~\pi(1)>1\}$.
Following~\cite[Eq.~(4.3)]{Roselle68}, one has
\begin{equation}\label{EGF1}
\sum_{n=0}^{\infty}P^*_n(x)\frac{z^n}{n!}=\frac{1-x}{\mathrm{e}^{xz}-x\mathrm{e}^{z}}.
\end{equation}
Comparing~\eqref{EGF1} with~\eqref{dxz-EGF}, one can immediately find that
\begin{equation}\label{Pndnx}
P^*_n(x)=d_n(x).
\end{equation}
The ascent polynomials over $\mqn$ are defined by $P_n(x)=\sum_{\pi\in\mqn}x^{\asc(\pi)+1}$.
Using~\cite[Eq.~(3.8)]{Roselle68}, we arrive at
\begin{equation}\label{Pnx-decom}
P_n(x)=P^*_n(x)+xP^*_{n-1}(x)=d_n(x)+xd_{n-1}(x).
\end{equation}
When $x=1$, it reduces to~\eqref{qndn}. As $d_n(x)$ is $\gamma$-positive, we immediately find the following result.
\begin{proposition}\label{Pnx-bi-gamma}
The polynomials $P_n(x)$ are bi-$\gamma$-positive.
\end{proposition}

A {\it drop} of $\pi\in\msn$ is an index $i$ such that $\pi(i)<i$.
Let $\drop(\pi)$ denote the number of drops of $\pi$.
For $\pi\in\mdn_n$, it is clear that $\exc(\pi)+\aexc(\pi)=n$.
The bivariate derangement polynomials are defined by $$d_n(x,y)=\sum_{\pi\in\mdn_n}x^{\exc(\pi)}y^{\aexc(\pi)}.$$
It follows from~\eqref{dxz-EGF} that
\begin{equation}\label{dxyz-EGF}
d(x,y;z)=\sum_{n=0}^\infty d_n(x,y)\frac{z^n}{n!}=\frac{y-x}{y\mathrm{e}^{xz}-x\mathrm{e}^{yz}}.
\end{equation}
Define
$$C_n(x,y,s)=\sum_{\pi\in\msn}x^{\exc(\pi)}y^{\aexc(\pi)}s^{\fix(\pi)}.$$
In particular, $d_n(x,y)=C_n(x,y,0)$ and $A_n(x)=C_n(x,1,1)$.
Since $$C_n(x,y,s)=\sum_{i=0}^n\binom{n}{i}s^id_{n-i}(x,y),$$ it follows that
\begin{equation}\label{dxys-EGF}
C(x,y,s;z)=\sum_{n=0}^{\infty}C_n(x,y,s)\frac{z^n}{n!}=\frac{(y-x)\mathrm{e}^{sz}}{y\mathrm{e}^{xz}-x\mathrm{e}^{yz}}.
\end{equation}
Define
$${A}_n(x,y)=C_n(x,y,y)=\sum_{\pi\in\msn}x^{\exc(\pi)}y^{\drop(\pi)+\fix(\pi)}.$$
We have
\begin{equation}\label{Anxy02}
\sum_{n=0}^\infty{A}_n(x,y)\frac{z^n}{n!}=\frac{(y-x)\mathrm{e}^{yz}}{y\mathrm{e}^{xz}-x\mathrm{e}^{yz}}.
\end{equation}
\begin{definition}
Let $\pi\in\msn$.
A {\it big ascent} of $\pi$ is an index $i\in [n-1]$ such that $\pi(i+1)\geqslant \pi(i)+2$.
The number of {\it big ascents} of $\pi$ is defined by
$\basc(\pi)=\#\{i\in [n-1]:~\pi(i+1)\geqslant \pi(i)+2\}$.
\end{definition}
It is clear that $\asc(\pi)=\suc(\pi)+\basc(\pi)$. When $\pi\in\mqn$, one has $\basc(\pi)=\asc(\pi)$.
We now define the following {\it trivariate Eulerian polynomials}
\begin{equation}\label{AnxysDef}
A_n(x,y,s)=\sum_{\pi\in\msn}x^{\basc(\pi)}y^{\des(\pi)}s^{\suc(\pi)}.
\end{equation}
In particular, $A_n(x)=A_n(x,1,x)=A_n(1,x,1)$.
Below are these polynomials for $n\leqslant 5$:
\begin{align*}
A_0(x,y,s)&=A_1(x,y,s)=1,~
A_2(x,y,s)=s+y,\\
A_3(x,y,s)&=(s+y)^2+2xy,~
A_4(x,y,s)=(s+y)^3+6xy(s+y)+2xy(x+y),\\
A_5(x,y,s)&=(s+y)^4+12xy(s+y)^2+8xy(s+y)(x+y)+2xy(x+y)^2+16x^2y^2.
\end{align*}
Define
$$A:=A(x,y,s;z)=\sum_{n=0}^{\infty}A_{n+1}(x,y,s)\frac{z^n}{n!}.$$
Note that $\des(\pi)=n-1-\suc(\pi)-\asc(\pi)$ for $\pi\in\msn$. Combining this with~\cite[Eq.~(5.9)]{Roselle68} and~\cite[Eq.~(6.9)]{Roselle68}, it is routine to deduce that
\begin{equation}\label{A-EGF}
A=\mathrm{e}^{z(y+s)}\left(\frac{y-x}{y\mathrm{e}^{xz}-x\mathrm{e}^{yz}}\right)^2,
\end{equation}
which can be verified by using~\eqref{recu-Anxys02}. In Corollary~\ref{Corend}, we give a generalization of~\eqref{A-EGF}.
It should be noted that~\eqref{A-EGF} can be seen as a special case of~\cite[Theorem~1]{Zeng1993}.

Comparing~\eqref{A-EGF} with~\eqref{dxyz-EGF},~\eqref{dxys-EGF} and~\eqref{Anxy02}, we get the following result.
\begin{proposition}
For $n\geqslant 0$, we have
\begin{align*}
A_{n+1}(x,y,s)&=\sum_{i=0}^n\binom{n}{i}{A}_i(x,y)C_{n-i}(x,y,s),\\
A_{n+1}(x,y,-y)&=\sum_{i=0}^n\binom{n}{i}{d}_i(x,y)d_{n-i}(x,y).
\end{align*}
In particular,
\begin{align}
A_{n+1}(x,1,0)&=\sum_{i=0}^n\binom{n}{i}A_i(x)d_{n-i}(x),\label{AD}\\
A_{n+1}(x,1,1)&=\sum_{i=0}^n\binom{n}{i}A_i(x)A_{n-i}(x),\label{AA}.
\end{align}
\end{proposition}

Note that $A_{n}(x,1,0)=\sum_{\pi\in\mqn}x^{\basc(\pi)}$.
By~\eqref{AD}, we get
\begin{equation*}\label{ADdecom}
A_{n+1}(x,1,0)=d_n(x)+\sum_{i=1}^{n}\binom{n}{i}A_i(x)d_{n-i}(x).
\end{equation*}
From~\eqref{Pndnx}, we see that $$d_n(x)=\sum_{\substack{\pi\in\mq_{n+1}\\ \pi(1)=1}}x^{\basc(\pi)}.$$
Thus we have
$$\sum_{i=1}^{n}\binom{n}{i}A_i(x)d_{n-i}(x)=\sum_{\substack{\pi\in\mq_{n+1}\\ \pi(1)>1}}x^{\basc(\pi)}.$$
Combining this with~\eqref{Pnx-decom}, we find the following result.
\begin{corollary}\label{corAnx10}
We have $A_n(x,1,0)$ is bi-$\gamma$-positive and
$$x\sum_{i=1}^{n}\binom{n}{i}A_i(x)d_{n-i}(x)=d_{n+1}(x).$$
\end{corollary}
%
%It is well known (see~\cite{Gal05}) that the product of two $\gamma$-positive polynomials is still $\gamma$-positive.
%Recall that $d_1(x)=0$. Combining~\eqref{Anx-gamma} and~\eqref{dnxq-gamma}, we find that $A_i(x)d_{n-i}(x)$ is $\gamma$-positive, it is divisible by $x$ and $\deg \left(A_i(x)d_{n-i}(x)\right)=n-2$, where $1\leqslant i\leqslant n-2$.
%Hence $\sum_{i=1}^{n}\binom{n}{i}A_i(x)d_{n-i}(x)$ is $\gamma$-positive, and so we obtain the following result.
%\begin{proposition}
%The polynomial $A_{n}(x,1,0)$ is bi-$\gamma$-positive, and so it is alternatingly increasing.
%\end{proposition}

Recall that $A_{n}(x,1,1)=\sum_{\pi\in\msn}x^{\basc(\pi)}$.
By~\eqref{AA}, we have
$$A_{n+1}(x,1,1)=2A_n(x)+\sum_{i=1}^{n-1}\binom{n}{i}A_i(x)A_{n-i}(x).$$
Clearly, $\deg \left(A_{n+1}(x,1,1)\right)=n-1$.
It is well known (see~\cite{Gal05}) that the product of two $\gamma$-positive polynomials is still $\gamma$-positive.
Using~\eqref{Anx-gamma}, we see that $A_i(x)A_{n-i}(x)$ is $\gamma$-positive and $\deg \left(A_i(x)A_{n-i}(x)\right)=n-2$ for any $1\leqslant i\leqslant n-1$.
Note that $A_i(x)A_{n-i}(x)$ is not divisible by $x$.
We end this section by giving the following result.
\begin{proposition}\label{Pnxspiral}
For any $n\geqslant 2$, the polynomial $x^{n-2}A_n(1/x,1,1)$ is bi-$\gamma$-positive, and so the big ascent polynomial $A_{n}(x,1,1)$ is spiral.
\end{proposition}
%%%%%%%%%%%%%%%%%%%%%%%%%%%5%%%%
%%%%%%%%%%%%%%%%%%%%%%%%%%%%%%%%%%%%%%%%%%%
\section{Triple and quadruple statistics}\label{section03}
%%%%%%%%%%%%%%%%%%%%%%%%%%%%
%\hspace*{\parindent}
%%\subsection{%%%%%%%%%%%%%%%%%%%%%%%%%%%%%%%%%%%%%%%%%%%%%%%%%%%
%%%%%%%%%%%%%%%%%%%%%%%%%%%%%%%%%%%%%%%%%%%%%%%%%%%%%%%%%%%
%%%%%%%%%%%%%%%%%%%%%%%%%%%5%%%%
%%%%%%%%%%%%%%%%%%%%%%%%%%%5%%%%
%%%%%%%%%%%%%%%%%%%%%%%%%%%%%%%%%%%%%%%%%%%
%%%%%%%%%%%%%%%%%%%%%%%%%%%%%%%%%%%%%%%%%%%
\subsection{Main results}
%%%%%%%%%%%%%%%%%%%%%%%%%%%%
\hspace*{\parindent}
%%\subsection{%%%%%%%%%%%%%%%%%%%%%%%%%%%%%%%%%%%%%%%%%%%%%%%%%%%
%%%%%%%%%%%%%%%%%%%%%%%%%%%%%%%%%%%%%%%%%%%%%%%%%%%%%%%%%%%
%%%%%%%%%%%%%%%%%%%%%%%%%%%5%%%%
%%%%%%%%%%%%%%%%%%%%%%%%%%%5%%%%
%%%%%%%%%%%%%%%%%%%%%%%%%%%%%%%%%%%%%%%%%%%

An {\it increasing tree} on $\{0,1,2,\ldots,n\}$ is a rooted tree
with vertex set $\{0,1,2,\ldots,n\}$ in which the labels of the vertices are increasing along any path from the root 0 to a leaf.
The {\it degree} of a vertex is referred to the number of its children.
A {\it 0-1-2 increasing tree} is an increasing tree in which the degree of any vertex is at most two.
\begin{definition}\label{def-forest}
A 0-1-2 increasing planted tree on $\{0,1,\ldots,n\}$ is a rooted tree with the root $0$ satisfying the following two conditions:
\begin{itemize}
  \item [$(i)$] the degree of each child of the root $0$ is at most one;
  \item [$(ii)$] the components of the root $0$ are vertex-disjoint 0-1-2 increasing trees and the union of the labels of these components forms a set partition of $[n]$.
\end{itemize}
\end{definition}
An illustration of a 0-1-2 increasing planted tree on $\{0,1,\ldots,8\}$ is given by Fig.~\ref{fig:simple}, where we assign a weight to each vertex
and there are three components of the root $0$.
\begin{figure}[t]\label{fig:simple}
\centering

\caption{The labeling of a 0-1-2 increasing planted tree on $\{0,1,2,\ldots,8\}$. \label{fig:simple}}
{
    \begin{tikzpicture}[->,>=stealth', level/.style={sibling distance = 5cm/#1, level distance = 2.0cm}, scale=0.6,transform shape]
    \node [treenode] {$0$~$(I)$}
    child
    {
        node [treenode] {$1$~$(1)$}
        child
            {
                node [treenode] {$7$~$(u)$}
            }}
    child
    {
        node [treenode] {$2$~(1)}
        child
        {
            node [treenode] {$3$~$(1)$}
            child
            {
                node [treenode] {$5$~$(u)$}
            }
            child
           {
                node [treenode] {$6$~$(v)$}
            child
            {
                node [treenode] {$8$~$(u)$}
            }}
        }
  %      child[edge from parent path ={(\tikzparentnode.-50) -- (\tikzchildnode.north)}]
%        {
%            node [treenode,yshift=0.4cm] (a) {}   % delay the text till later
%        }
    }
    child[edge from parent path ={(\tikzparentnode.-30) -- (\tikzchildnode.north)}]
    {
        node [treenode,yshift=0.4cm] (b) {}       % delay the text till later
    }
;
% ------------------------------------------------ put the text into subtree nodes
%\node[align=center,yshift=0.1cm] at (a) {$5$~$(u)$};
\node[align=center,yshift=0.1cm] at (b) {$4$~$(t)$};
\end{tikzpicture}
}
\end{figure}
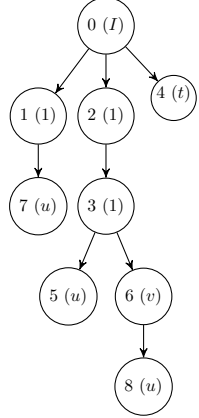

We can now present the first main result of this paper.
\begin{theorem}\label{mainthm01}
Let $A_n(x,y,s)$ be the trivariate Eulerian polynomials defined by~\eqref{AnxysDef}. Then
\begin{equation}\label{Anxys-recu}
A_{n+1}(x,y,s)=(s+y)A_{n}(x,y,s)+xy\left(\frac{\partial}{\partial x}+\frac{\partial}{\partial y}+\frac{\partial}{\partial s}\right)A_{n}(x,y,s),
\end{equation}
which can be rewritten as
\begin{equation}\label{recu-Anxys02}
\frac{\partial}{\partial z}A=(s+y)A+xy\left(\frac{\partial}{\partial x}+\frac{\partial}{\partial y}+\frac{\partial}{\partial s}\right)A.
\end{equation}
Moreover, one has
\begin{equation}\label{Anxys-gamma}
A_{n+1}(x,y,s)=\sum_{i=0}^n(s+y)^i\sum_{j=0}^{\lrf{(n-i)/2}}\gamma_{n,i,j}(2xy)^j(x+y)^{n-i-2j},
\end{equation}
where the coefficient $\gamma_{n,i,j}$ satisfies the recursion
\begin{equation}\label{gammanij-recu}
\gamma_{n+1,i,j}=\gamma_{n,i-1,j}+(1+i)\gamma_{n,i+1,j-1}+j\gamma_{n,i,j}+(n-i-2j+2)\gamma_{n,i,j-1},
\end{equation}
with the initial conditions $\gamma_{0,0,0}=1$ and $\gamma_{0,i,j}=0$ for $(i,j)\neq (0,0)$. The number $\gamma_{n,i,j}$ equals the
number of 0-1-2 increasing planted trees on $\{0,1,\ldots,n\}$ with
$i+j$ leaves, among which $i$ leaves are children of the root.
\end{theorem}

Combining~\eqref{dxz-EGF} and~\eqref{A-EGF}, we see that
$$A_{n+1}(x,1,-1)=\sum_{i=0}^n\binom{n}{i}{d}_i(x)d_{n-i}(x).$$
\begin{corollary}
We have
\begin{equation*}
A_{n+1}(x,y,-y)=\sum_{\pi\in\ms_{n+1}}x^{\basc(\pi)}y^{\des(\pi)}(-y)^{\suc(\pi)}=\sum_{j=0}^{\lrf{n/2}}\gamma_{n,0,j}(2xy)^j(x+y)^{n-2j},
\end{equation*}
and so the binomial convolution of the derangement polynomials are $\gamma$-positive, i.e.,
\begin{equation*}
\sum_{i=0}^n\binom{n}{i}{d}_i(x)d_{n-i}(x)=A_{n+1}(x,1,-1)=\sum_{j=0}^{\lrf{n/2}}\gamma_{n,0,j}(2x)^j(1+x)^{n-2j}.
\end{equation*}
\end{corollary}

Given any $\pi\in\msn$, we define
\begin{align*}
&\rm{bAsc}(\pi)=\{\pi(i+1): \pi(i+1)\geqslant \pi(i)+2,~i\in [n-1]\},\\
&\rm{Des}(\pi)=\{\pi(i+1): \pi(i)>\pi(i+1),~i\in[n-1]\},\\
&\rm{Suc}(\pi)=\{\pi(i+1): \pi(i+1)=\pi(i)+1,~i\in[n-1]\},\\
&\Aexc(\pi)=\{\pi(i): \pi(i)<i,~i\in\{2,3,\ldots,n\}\},\\
&\widehat{\Exc}(\pi)=\{\pi(i): \pi(i)>i,~i\in\{2,3,\ldots,n\}\},\\
&\widehat{\Fix}(\pi)=\{\pi(i): \pi(i)=i,~i\in\{2,3,\ldots,n\}\}.
\end{align*}
Set $\aexc(\pi)=\#\Aexc(\pi),~\widehat{\exc}(\pi)=\#\widehat{\Exc}(\pi)$ and $\widehat{\fix}(\pi)=\#\widehat{\Fix}(\pi)$.
\begin{theorem}\label{NewEu}
The following two triple set-valued statistics are equidistributed over $\msn$:
$$\left(\rm{bASC},~\rm{Des},~\rm{Suc}\right),~\left(\widehat{\Exc},~\Aexc,~\widehat{\Fix}\right).$$
So we have
$$\sum_{\pi\in\msn}x^{\basc(\pi)}y^{\des(\pi)}s^{\suc(\pi)}=\sum_{\pi\in\msn}x^{\widehat{\exc}(\pi)}y^{\aexc(\pi)}s^{\widehat{\fix}(\pi)}.$$
Since $\widehat{\exc}+\widehat{\fix}$ is equidistributed with $\asc$ over $\msn$, it is an Eulerian statistic.
\end{theorem}
%%%%%%%%%%%%%%%%%%%%%%%%%%%%%%%%%%%%%%%%%%%
\subsection{Proof of Theorem~\ref{mainthm01}}
%%%%%%%%%%%%%%%%%%%%%%%%%%%%
\hspace*{\parindent}
%%\subsection{%%%%%%%%%%%%%%%%%%%%%%%%%%%%%%%%%%%%%%%%%%%%%%%%%%%
%%%%%%%%%%%%%%%%%%%%%%%%%%%%%%%%%%%%%%%%%%%%%%%%%%%%%%%%%%%
%%%%%%%%%%%%%%%%%%%%%%%%%%%5%%%%

For an alphabet $A$, let $\mathbb{Q}[[A]]$ be the rational commutative ring of formal power
series in monomials formed from letters in $A$. Following Chen~\cite{Chen93}, a {\it context-free grammar} over
$A$ is a function $G: A\rightarrow \mathbb{Q}[[A]]$ that replaces each letter in $A$ by a formal function over $A$.
The formal derivative $D_G$ with respect to $G$ satisfies the derivation rule:
$$D_G(u+v)=D_G(u)+D_G(v),~D_G(uv)=D_G(u)v+uD_G(v).$$
In the theory of context-free grammars, there are two widely used method.
Following~\cite{Chen17,Dumont96}, the {\it grammatical labeling method} is an assignment of the underlying elements of
a combinatorial structure with variables, which is consistent with the substitution rules
of a grammar.
Another well known method is the {\it change of grammars}, which essentially is a change of variables, see~\cite{Chen22,Chen23,Lin21,Ma19,Ma2401,Ma23} for applications.

The following result is fundamental.
%%%%%%%%%%%%%%%%%%%%%%%%%%%%%%%%%%%%%%%%%%%
\begin{lemma}\label{lemmaLM}
If
\begin{equation}\label{LMxys-G}
G=\{L\rightarrow Ly,M\rightarrow Ms,s\rightarrow xy,x\rightarrow xy,y\rightarrow xy\},
\end{equation}
then we have
\begin{equation}\label{derangment-grammarLMA}
D_{G}^n(LM)=LMA_{n+1}(x,y,s)=LM\sum_{\pi\in\ms_{n+1}}x^{\basc(\pi)}y^{\des(\pi)}s^{\suc(\pi)}.
\end{equation}
\end{lemma}
\begin{proof}
We first introduce a grammatical labeling of $\pi=\pi(1)\pi(2)\cdots\pi(n)\in \msn$:
\begin{itemize}
\item [\rm ($i$)]Put a superscript label $L$ at the front of $\pi$;
\item [\rm ($ii$)]Put a superscript label $M$ right after the maximum entry $n$;
  \item [\rm ($iii$)]If $i$ is a big ascent, then put a superscript label $x$ right after $\pi(i)$;
 \item [\rm ($iv$)]If $i$ is a descent and $\pi(i)\neq n$, then put a superscript label $y$ right after $\pi(i)$;
  \item [\rm ($v$)]If $\pi(n)\neq n$, then put a superscript label $y$ at the end of $\pi$;
\item [\rm ($vi$)]If $i$ is a succession, then put a superscript label $s$ right after $\pi(i)$.
\end{itemize}
The weight of $\pi$ is defined to be the product of its labels.
Thus the weight of $\pi$ is given by $$w(\pi)=LMx^{\basc(\pi)}y^{\des(\pi)}s^{\suc(\pi)}.$$
Note that $\ms_1=\{^L1^M\}$ and $\ms_2=\{^L1^s2^M, ^L2^M1^y\}$. Note that $D_{G}(LM)=LM(s+y)$.
The weight of the element in $\ms_1$ is $LM$ and the sum of weights of the elements in $\ms_2$ is given by $D_{G}(LM)$.
Suppose we get all labeled permutations in $\ms_{n-1}$, where $n\geqslant 2$. Let
$\widehat{{\pi}}$ be a permutation obtained from $\pi\in\ms_{n-1}$ by inserting $n$.
There are six cases to label $n$ and relabel some elements of $\pi$. Setting $\pi_i=\pi(i)$,
then the changes of labeling can be illustrated as follows:
$$ ^L\pi_1\cdots (n-1)^M\cdots   \mapsto  ^Ln^M\pi_1\cdots (n-1)^y\cdots ;$$
$$ ^L\pi_1\cdots (n-1)^M\cdots   \mapsto ^L\pi_1\cdots (n-1)^sn^M\cdots;$$
$$ \cdots\pi_i^x\cdots (n-1)^M\cdots   \mapsto \cdots\pi_i^xn^M\cdots (n-1)^y\cdots ;$$
$$ \cdots\pi_i^y\pi_{i+1}\cdots(n-1)^M\cdots   \mapsto \cdots\pi_i^xn^M\pi_{i+1}\cdots(n-1)^y\cdots;$$
$$ \cdots (n-1)^M\cdots\pi_{n-1}^y   \mapsto  \cdots (n-1)^y\cdots\pi_{n-1}^xn^M;$$
$$ \cdots \pi_i^s\pi_{i+1}\cdots(n-1)^M\cdots   \mapsto \cdots \pi_i^xn^M\pi_{i+1}\cdots (n-1)^y\cdots.$$
In each case, the insertion of $n$ corresponds to one substitution rule in $G$. By induction,
it is routine to check that the action of the formal derivative $D_{G}$ on the set of weighted permutations in $\ms_{n-1}$ gives the set of weighted permutations in
$\ms_{n}$. This completes the proof of~\eqref{derangment-grammarLMA}.
\end{proof}
%%%%%%%%%%%%%%%%%%%%%%%%%%%%%%%%%%%%%%%%%%%%%%%%%%%%%%%%%%%
%%%%%%%%%%%%%%%%%%%%%%%%%%%%%%%%%%%%%%%%%%%
%%%%%%%%%%%%%%%%%%%%%%%%%%%%%%%%%%%%%%%%%%%%%%%%%%%%%%%%%%%%
%%%%%%%%%%%%%%%%%%%%%%%%%%%%%%%%%%%%%%%%%%%%%
%\subsection{Proof of Theorem~\ref{mainthm01}}
%%%%%%%%%%%%%%%%%%%%%%%%%%%%%%
%\hspace*{\parindent}
%%%\subsection{%%%%%%%%%%%%%%%%%%%%%%%%%%%%%%%%%%%%%%%%%%%%%%%%%%%
%%%%%%%%%%%%%%%%%%%%%%%%%%%%%%%%%%%%%%%%%%%%%%%%%%%%%%%%%%%%
%%%%%%%%%%%%%%%%%%%%%%%%%%%%%%%%%%%%%%%%%%%%
\noindent{\bf A proof
Theorem~\ref{mainthm01}:}
\begin{proof}
\quad (A) Let $G$ be the grammar defined by~\eqref{LMxys-G}.
By induction, we see that there exist nonnegative integers $a_{n,i,j}$ such that
$$D_{G}^n(LM)=LM\sum_{i,j=0}^na_{n,i,j}x^iy^js^{n-i-j}.$$
Then we get
\begin{align*}
&D_{G}\left(D_{G}^n(LM)\right)\\
&=LM\sum_{i,j=0}^na_{n,i,j}\left(x^iy^{j+1}s^{n-i-j}+x^iy^js^{n+1-i-j}\right)+\\
&LM\sum_{i,j=0}^na_{n,i,j}\left(ix^iy^{j+1}s^{n-i-j}+jx^{i+1}y^{j}s^{n-i-j}+(n-i-j)x^{i+1}y^{j+1}s^{n-1-i-j}\right).
\end{align*}
Comparing the coefficients of $LMx^iy^js^{n+1-i-j}$ in both sides of the above expression, we get
\begin{equation}\label{anij-recu}
a_{n+1,i,j}=a_{n,i,j}+(1+i)a_{n,i,j-1}+ja_{n,i-1,j}+(n-i-j+2)a_{n,i-1,j-1}.
\end{equation}
Multiplying both sides of~\eqref{anij-recu} by $x^iy^js^{n+1-i-j}$ and summing over all $i,j$,
we arrive at~\eqref{Anxys-recu}.

\quad (B) We now make a change of variables. Setting
$u=2xy,v=x+y,t=s+y$ and $I=LM$,
we get
$D_{G}(u)=uv,D_{G}(v)=u,D_{G}(t)=u$ and $D_{G}(I)=It$. Thus we get a new grammar
\begin{equation}\label{JAcobi-gram02}
G'=\{I\rightarrow It,t\rightarrow u, u\rightarrow uv,v\rightarrow u\}.
\end{equation}
Note that $D_{G'}(I)=It,~D_{G'}^2(I)=I(t^2+u)$ and $D_{G'}^3(I)=I(t^3+3tu+uv)$.
Then by induction, it is routine to check that there exist nonnegative integers $\gamma_{n,i,j}$ such that
\begin{equation}\label{DG101}
D_{G'}^n(I)=I\sum_{i=0}^nt^i\sum_{j=0}^{\lrf{(n-i)/2}}\gamma_{n,i,j}u^jv^{n-i-2j}.
\end{equation}
Then upon taking $u=2xy,v=x+y,t=s+y$ and $I=LM$, we get~\eqref{Anxys-gamma}.
In particular, $\gamma_{0,0,0}=1$ and $\gamma_{0,i,j}=0$ if $(i,j)\neq (0,0)$.
Since $D_{G'}^{n+1}(I)=D_{G'}\left(D_{G'}^n(I)\right)$,
we obtain
\begin{align*}
D_{G'}\left(D_{G'}^n(I)\right)&
%&=D_{G_1}\left(I\sum_{i=0}^nt^i\sum_{j=0}^{\lrf{(n-i)/2}}2^j\gamma_{n,i,j}u^jv^{n-i-2j}\right)\\
=I\sum_{i,j}\gamma_{n,i,j}\left(t^{i+1}u^jv^{n-i-2j}+it^{i-1}u^{j+1}v^{n-i-2j}\right)+\\
&I\sum_{i,j}\gamma_{n,i,j}\left(jt^iu^jv^{n+1-i-2j}+(n-i-2j)t^iu^{j+1}v^{n-1-i-2j}\right).
\end{align*}
Comparing the coefficients of $t^iu^jv^{n+1-i-2j}$ in both sides of the above expansion, we get~\eqref{gammanij-recu}.

\quad (C) The combinatorial interpretation of $\gamma_{n,i,j}$ can be found by using the following grammatical labeling.
Given a 0-1-2 increasing planted tree $T$, the root 0 is labeled by $I$. For the children of the root,
each child with degree 0 (a leaf of the root) is labeled by $t$ and each child with degree one is labeled by 1.
For the other vertices (not the children of the root), each leaf is labeled by $u$, each vertex with degree one is labeled by $v$ and each vertex of degree two is labeled by 1.
See Fig.~\ref{fig:simple} for an example, where
the grammatical labels are given in parentheses.

Let $T$ be the 0-1-2 increasing planted tree given in Fig.~\ref{fig:simple}.
We distinguish four cases:
\begin{itemize}
\item [\rm ($i$)]If we add 9 as a child of the root 0, then the vertex 9 becomes a leaf of the root, and the label of 9 is $t$.
This corresponds to the substitution rule $I\rightarrow It$;
\item [\rm ($ii$)]If we add 9 as a child of the vertex 4, the label of 4 becomes 1, and the vertex 9 gets the label $u$.
This corresponds to the substitution rule $t\rightarrow u$;
  \item [\rm ($iii$)]If we add 9 as a child of the vertex 5 (resp.~7, 8), the label $u$ of 5 (resp.~7, 8) becomes $v$, and the vertex 9 gets the label $u$.
This corresponds to the substitution rule $u\rightarrow uv$;
 \item [\rm ($iv$)]If we add 9 as a child of the vertex 6, the label $v$ of 6 becomes 1, and the vertex 9 gets the label $u$.
This corresponds to the substitution rule $v\rightarrow u$.
\end{itemize}

The aforementioned four cases exhaust all the cases to construct a 0-1-2 increasing planted tree $T'$ on $\{0,1,2,\ldots,n,n+1\}$
from a 0-1-2 increasing planted tree $T$ on
$\{0,1,2,\ldots,n\}$ by adding $n+1$ as a leaf.
Since $D_{G'}^n(I)$ equals the sum of the weights
of 0-1-2 increasing planted trees on $\{0,1,2,\ldots,n\}$, then $\gamma_{n,i,j}$
counts 0-1-2 increasing planted tree $T$ on $\{0,1,2,\ldots,n\}$ with $i+j$ leaves, among which $i$ leaves are the children of the root. This completes the proof.
\end{proof}
%%%%%%%%%%%%%%%%%%%%%%%%%%%%%%%%%%%%%%%%%%%
\subsection{Proof of Theorem~\ref{NewEu}}
%%%%%%%%%%%%%%%%%%%%%%%%%%%%
\hspace*{\parindent}
%%\subsection{%%%%%%%%%%%%%%%%%%%%%%%%%%%%%%%%%%%%%%%%%%%%%%%%%%%
%%%%%%%%%%%%%%%%%%%%%%%%%%%%%%%%%%%%%%%%%%%%%%%%%%%%%%%%%%%
%%%%%%%%%%%%%%%%%%%%%%%%%%%5%%%%

We now write any permutation in $\msn$ by using its standard cycle form, where each cycle
is written with its smallest entry first and the cycles are written in increasing order of
their smallest entry. Another grammatical labeling of $\pi\in \msn$ is given as follows:
\begin{itemize}
\item [\rm ($i$)]Put a superscript label $L$ right after the entry $1$;
\item [\rm ($ii$)]Put a superscript label $M$ at the end of $\pi$;
  \item [\rm ($iii$)]If $\pi(i)\in \Aexc(\pi)$, then put a superscript label $y$ right after $i$;
 \item [\rm ($iv$)]If $\pi(i)\in \widehat{\Exc}(\pi)$, then put a superscript label $x$ right after $i$;
  \item [\rm ($v$)]If $\pi(i)\in \widehat{\Fix}(\pi)$, then put a superscript label $s$ right after $i$.
\end{itemize}
%The weight of $\pi$ is defined to be the product of its labels.
Thus the weight of $\pi$ is given by $$w(\pi)=LMx^{\widehat{\exc}(\pi)}y^{\aexc(\pi)}s^{\widehat{\fix}(\pi)}.$$
In particular, $\ms_1=\{(1^L)^M\},~\ms_2=\{(1^L)(2^s)^M,~(1^L,2^y)^M\}$, and the elements in $\ms_3$ are listed as follows:
$$(1^L)(2^s)(3^s)^M,~(1^L)(2^x,3^y)^M,~(1^L,3^y)(2^s)^M,~(1^L,2^y)(3^s)^M,~(1^L,3^y,2^y)^M,~(1^L,2^x,3^y)^M.$$
Along the same lines as in the proof of Lemma~\ref{lemmaLM}, one can easily deduce that
\begin{equation}\label{derangment-grammarLMA02}
D_{G}^n(LM)=LM\sum_{\pi\in\ms_{n+1}}x^{\widehat{\exc}(\pi)}y^{\aexc(\pi)}s^{\widehat{\fix}(\pi)}.
\end{equation}
As illustrated by Example~\ref{exampleLM},
by analyzing the changes of labeling, it is routine to check that
$$\left(\rm{bASC},~\rm{Des},~\rm{Suc}\right),~\left(\widehat{\Exc},~\Aexc,~\widehat{\Fix}\right)$$
are equidistributed over $\msn$ and we omit the details for simplicity.

\begin{example}\label{exampleLM}
When $n=3$, the correspondences of $\left(\rm{bASC},~\rm{Des},~\rm{Suc}\right)$ and $\left(\widehat{\Exc},~\Aexc,~\widehat{\Fix}\right)$ can be listed as follows:
\begin{align*}
&^L1^s2^s3^M\leftrightarrow (1^L)(2^s)(3^s)^M;~^L1^x3^M2^y\leftrightarrow (1^L)(2^x3^y)^M;~^L3^M1^s2^y\leftrightarrow (1^L3^y)(2^s)^M;\\
&^L2^y1^x3^M\leftrightarrow (1^L2^x3^y)^M;~^L2^s3^M1^y\leftrightarrow (1^L2^y)(3^s)^M;~^L3^M2^y1^y\leftrightarrow (1^L3^y2^y)^M.
\end{align*}
\end{example}

%%%%%%%%%%%%%%%%%%%%%%%%%%%%%%%%%%%%%%%%%%%%%%%%%%%%%%%%%%%%%%%%%%%%%%%%%%%%%%%%%%
\subsection{Another interpretation of the coefficients $\gamma_{n,i,j}$}
%%%%%%%%%%%%%%%%%%%%%%%%%%%%%%%%%%%%%%%%%%%
\hspace*{\parindent}

Simsun permutations were introduced by Simion and Sundaram when they studied the action of the symmetric group on the maximal chains of the
partition lattice~\cite[p.~267]{Sundaram1994}.
We say that $\pi\in\msn$ has no {\it proper double descents} if there is no
index $i\in [n-2]$ such that $\pi(i)>\pi(i+1)>\pi(i+2)$.
Then $\pi$ is called {\it simsun} if for all $k$, the
subword of $\pi$ restricted to $[k]$ (in the order
they appear in $\pi$) contains no proper double descents.
Let $\rs_n$ be the set of simsun permutations of length $n$.
Define $$S_n(x)=\sum_{\pi\in\rs_n}x^{\des(\pi)}.$$
Here we list another three combinatorial interpretations of $S_n(x)$:
\begin{itemize}
  \item the polynomial $S_n(x)$ is also the descent polynomial of And\'re permutations of the second kind of order $n+1$, see~\cite{Chow11,Foata01};
  \item the polynomial $xS_n(x)$ equals the Andr\'e polynomial that counts increasing 0-1-2 trees on $[n+1]$ by their leaves, see~\cite{Chen17,Foata01};
  \item the polynomial $S_n(x)$ counts simsun permutations of the second kind of order $n$ by their numbers of excedances, see~\cite{Ma17}.
\end{itemize}

A value $x=\pi(i)$ is called a {\it cycle double ascent} of $\pi$ if
$i=\pi^{-1}(x)<x<\pi(x)$.
We say that $\pi\in\msn$ is a {\it simsun permutation of the second kind} if for all $k\in [n]$,
after removing the $k$ largest letters of $\pi$, the resulting permutation has no cycle double ascents.
For example, $(1,6,5,3,4)(2)$ is not a simsun permutation of the second kind since when we remove the letters 5 and 6,
the resulting permutation $(1,3,4)(2)$ contains the cycle double ascent 3.
Let $\rss_n$ be the set of the simsun permutations of the second kind of length $n$.
We can now present another interpretation of the coefficients $\gamma_{n,i,j}$ defined by~\eqref{Anxys-gamma}.
\begin{proposition}\label{combinatorial-inter}
For any $0\leqslant i\leqslant n$ and $0\leqslant j\leqslant \lrf{(n-i)/2}$,
the number $\gamma_{n,i,j}$ counts simsun permutations of the second kind of order $n$ which have exactly $i$ fixed points and $j$ excedances.
\end{proposition}
\begin{proof}
We write any permutation in $\rss_n$ by using its standard cycle form.
In order to get a permutation $\pi'\in\rss_{n+1}$ with $i$ fixed points and $j$ excedances from a permutation $\pi\in\rss_{n}$,
we distinguish four cases:
\begin{enumerate}
\item [($c_1$)] If $\pi\in \rss_n$ and $\fix(\pi)=i-1$ and $\exc(\pi)=j$, then we need append $(n+1)$ to $\pi$ as a new cycle. This accounts for $\gamma_{n,i-1,j}$ possibilities;
\item [($c_2$)] If $\pi\in \rss_n$ and $\fix(\pi)=i+1$ and $\exc(\pi)=j-1$, then we should insert the entry $n+1$ right after a fixed point. This accounts for $(1+i)\gamma_{n,i+1,j-1}$ possibilities;
 \item [($c_3$)] If $\pi\in \rss_n$ and $\fix(\pi)=i$ and $\exc(\pi)=j$, then we should insert the entry $n+1$ right after an excedance.
This accounts for $j\gamma_{n,i,j}$ possibilities;
 \item [($c_4$)] Since $\pi\in\rss_n$ has no cycle double ascents, we say that $\pi(i)$ is a {\it cycle peak} if $i$ is an excedance, i.e. $i<\pi(i)$. If $\pi\in \rss_n$ and $\fix(\pi)=i$ and $\exc(\pi)=j-1$, then there are $n-i-2(j-1)$ positions could be inserted the entry $n+1$, since
we cannot insert $n+1$ immediately before or right after each cycle peak of $\pi$, and we cannot insert $n+1$ right after a fixed point.
This accounts for $(n-i-2j+2)\gamma_{n,i,j-1}$ possibilities.
 \end{enumerate}
Thus the recursion~\eqref{gammanij-recu} holds. This completes the proof.
\end{proof}

%%%%%%%%%%%%%%%%%%%%%%%%%%%%%%%%%%%%%%%%%%%
\subsection{Proper left-to-right minimum statistic}\label{section04}
%%%%%%%%%%%%%%%%%%%%%%%%%%%%
%\hspace*{\parindent}
%%\subsection{%%%%%%%%%%%%%%%%%%%%%%%%%%%%%%%%%%%%%%%%%%%%%%%%%%%
%%%%%%%%%%%%%%%%%%%%%%%%%%%%%%%%%%%%%%%%%%%%%%%%%%%%%%%%%%%
%%%%%%%%%%%%%%%%%%%%%%%%%%%5%%%%
%%%%%%%%%%%%%%%%%%%%%%%%%%%5%%%%
%%%%%%%%%%%%%%%%%%%%%%%%%%%%%%%%%%%%%%%%%%%
%%%%%%%%%%%%%%%%%%%%%%%%%%%%
\hspace*{\parindent}
%%\subsection{%%%%%%%%%%%%%%%%%%%%%%%%%%%%%%%%%%%%%%%%%%%%%%%%%%%
%%%%%%%%%%%%%%%%%%%%%%%%%%%%%%%%%%%%%%%%%%%%%%%%%%%%%%%%%%%
%%%%%%%%%%%%%%%%%%%%%%%%%%%5%%%%

Let $\pi=\pi(1)\pi(2)\cdots\pi(n)\in\msn$.
In this subsection, we always identify $\pi$ with the word $\pi(1)\pi(2)\cdots\pi(n)\pi(n+1)$, where $\pi(n+1)=0$.
For $1\leqslant i\leqslant n$,
a value $\pi(i)$ is called a {\it left-to-right minimum} if $\pi(i)<\pi(j)$ for all $1\leqslant j<i$ or $i=1$.
Let $\lrm(\pi)$ be the number of left-to-right minima of $\pi$.
\begin{definition}
Given $\pi\in\msn$. We say that $\pi(i)$ is a proper left-to-right minimum if it satisfies the following two conditions:
\begin{itemize}
  \item $\pi(i)$ is a left-to-right minimum and $\pi(i)\neq 1$,
  \item  there exists an index $k>i$ such that $\pi(k)=\pi(i)-1$ and $\pi(k)>\pi(k+1)$.
\end{itemize}
\end{definition}
Let $\plrm(\pi)$ be the number of proper left-to-right minima of $\pi$.
\begin{example}
For $\pi\in\ms_3$, we have $$\plrm(123)=\plrm(132)=\plrm(213)=0,$$
$$\plrm(\textbf{2}31)=\plrm(\textbf{3}12)=1,~\plrm(\textbf{32}1)=2.$$
\end{example}
Consider the $\suc$ and $\plrm$ $(s,t)$-Eulerian polynomials
$$A_n(x,y,s,t)=\sum_{\pi\in\msn}x^{\basc(\pi)}y^{\des(\pi)-\plrm(\pi)}s^{\suc(\pi)}t^{\plrm(\pi)}.$$
In particular, $A_1(x,y,s,t)=1,~A_2(x,y,s,t)=s+t,~A_3(x,y,s,t)=(s+t)^2+2xy$.
\begin{lemma}\label{lemmaLMSt}
If
\begin{equation*}\label{LMxyst-G}
G=\{L\rightarrow Lt,M\rightarrow Ms,s\rightarrow xy, t\rightarrow xy, x\rightarrow xy,y\rightarrow xy\},
\end{equation*}
then we have
\begin{equation}\label{derangment-grammarLMAst}
D_{G}^n(LM)=LMA_{n+1}(x,y,s,t)=LM\sum_{\pi\in\ms_{n+1}}x^{\basc(\pi)}y^{\des(\pi)-\plrm(\pi)}s^{\suc(\pi)}t^{\plrm(\pi)}.
\end{equation}
\end{lemma}
\begin{proof}
A grammatical labeling of $\pi=\pi(1)\pi(2)\cdots\pi(n)\in \msn$ can be given as follows:
\begin{itemize}
\item [\rm ($i$)]Put a superscript label $L$ at the front of $\pi$;
\item [\rm ($ii$)]Put a superscript label $M$ right after the maximum entry $n$;
  \item [\rm ($iii$)]If $i$ is a big ascent, then put a superscript label $x$ right after $\pi(i)$;
 \item [\rm ($iv$)]If $i$ is a succession, then put a superscript label $s$ right after $\pi(i)$;
  \item [\rm ($v$)] If $i$ is a descent, $\pi(i)\neq n$ and $\pi(i)+1$ is a left-to-right minimum, then put a superscript label $t$ right after $\pi(i)$;
    \item [\rm ($vi$)] If $i$ is a descent, $\pi(i)\neq n$ and $\pi(i)+1$ is not a left-to-right minimum, then put a superscript label $y$ right after $\pi(i)$.
\end{itemize}
Then the weight of $\pi$ is given by $$w(\pi)=LMx^{\basc(\pi)}y^{\des(\pi)-\plrm(\pi)}s^{\suc(\pi)}t^{\plrm(\pi)}.$$
Note that $\ms_1=\{^L1^M\}$ and $\ms_2=\{^L1^s2^M, ^L2^M1^t\}$. Note that $D_{G}(LM)=LM(s+t)$.
Hence the weight of the element in $\ms_1$ is $LM$ and the sum of weights of the elements in $\ms_2$ is given by $D_{G}(LM)$.
Along the same lines as in the proof of Lemma~\ref{lemmaLM}, one can discuss the general cases and we omit the details for simplicity.
\end{proof}
Define 
\begin{align*}
&\rm{Suc^*}(\pi)=\left\{\pi(i): \pi(i+1)=\pi(i)+1,~i\in[n-1]\right\},\\
\rm{Plrmin}(\pi)&=\{\pi(i): \text{$i$ is a descent, $\pi(i)\neq n$ and $\pi(i)+1$ is a left-to-right minimum}\}.
\end{align*}
Using the grammatical labeling given in the proof of Lemma~\ref{lemmaLMSt}, it is routine to verify the following result. An illustration of it is given by Example~\ref{exsuc}.
\begin{proposition}
The pair of set-valued statistics $(\rm{Suc^*},~\rm{Plrmin})$ is symmetric over $\msn$.
\end{proposition}
\begin{example}\label{exsuc}
Recall that $\ms_2=\{^L1^s2^M, ^L2^M1^t\}$. Consider the insertion of the entry $3$. Using the correspondences $L\leftrightarrow M$, $s\leftrightarrow t$ and $x\leftrightarrow y$,
the symmetry of the pair of set-valued statistics $(\rm{Suc^*},~\rm{Plrmin})$ is demonstrated as follows: 
\begin{align*}
&\left(\rm{Suc^*}(^L1^s2^s3^M),~\rm{Plrmin}(^L1^s2^s3^M)\right)=(\{1,2\},\emptyset)\leftrightarrow \left(\rm{Suc^*}(^L3^M2^t1^t),~\rm{Plrmin}(^L3^M2^t1^t)\right)=(\emptyset,\{1,2\});\\
&\left(\rm{Suc^*}(^L1^x3^M2^y),~\rm{Plrmin}(^L1^x3^M2^y)\right)=(\emptyset,\emptyset)\leftrightarrow \left(\rm{Suc^*}(^L2^y1^x3^M),~\rm{Plrmin}(^L2^y1^x3^M)\right)=(\emptyset,\emptyset);\\
&\left(\rm{Suc^*}(^L3^M1^s2^t),~\rm{Plrmin}(^L3^M1^s2^t)\right)=(\{1\},\{2\})\leftrightarrow \left(\rm{Suc^*}(^L2^s3^M1^t),~\rm{Plrmin}(^L2^s3^M1^t)\right)=(\{2\},\{1\}).
\end{align*}
\end{example}

The following theorem is easily derived from Lemma~\ref{lemmaLMSt} in the same way as Theorem~\ref{mainthm01}.
\begin{theorem}
For the $\suc$ and $\plrm$ $(s,t)$-Eulerian polynomials,
we have
\begin{equation}\label{Anxyst-recu}
A_{n+1}(x,y,s,t)=(s+t)A_{n}(x,y,s,t)+xy\left(\frac{\partial}{\partial x}+\frac{\partial}{\partial y}+\frac{\partial}{\partial s}+\frac{\partial}{\partial t}\right)A_{n}(x,y,s,t);
\end{equation}
\begin{equation}\label{Anxyst-gamma}
A_{n+1}(x,y,s,t)=\sum_{i=0}^n(s+t)^i\sum_{j=0}^{\lrf{(n-i)/2}}\gamma_{n,i,j}(2xy)^j(x+y)^{n-i-2j},
\end{equation}
which implies that $A_{n+1}(x,y,s,t)$ is symmetric in the variables $s$ and $t$ as well as $x$ and $y$.
\end{theorem}
\begin{corollary}
We have
\begin{equation*}
\sum_{\pi\in\ms_{n+1}}x^{\basc(\pi)}y^{\des(\pi)-\plrm(\pi)}s^{\suc(\pi)}(-s)^{\plrm(\pi)}=\sum_{j=0}^{\lrf{n/2}}\gamma_{n,0,j}(2xy)^j(x+y)^{n-2j}.
\end{equation*}
\end{corollary}
A special case of~\eqref{Anxyst-gamma} says that $A_{n+1}(x,1,s,t)$ is partial $\gamma$-positive, i.e.,
$$A_{n+1}(x,1,s,t)=\sum_{i=0}^n(s+t)^i\sum_{j=0}^{\lrf{(n-i)/2}}\gamma_{n,i,j}(2x)^j(x+1)^{n-i-2j}.$$
Comparing~\eqref{Anxyst-gamma} with Corollary~\ref{corAnx10}, we see that $A_{n}(x,1,s,t)$ is bi-$\gamma$-positive if $s+t=1$.
Combining this with~\cite[Theorem~2.4]{Ma2401}, we obtain the following result.
\begin{corollary}
Let $s$ and $t$ be given real numbers such that $0\leqslant s+t\leqslant 1$, then
$A_{n}(x,1,s,t)$ is alternatingly increasing.
\end{corollary}
%%\subsection{%%%%%%%%%%%%%%%%%%%%%%%%%%%%%%%%%%%%%%%%%%%%%%%%%%%
%%%%%%%%%%%%%%%%%%%%%%%%%%%%%%%%%%%%%%%%%%%%%%%%%%%%%%%%%%%
%%%%%%%%%%%%%%%%%%%%%%%%%%%5%%%%
%%%%%%%%%%%%%%%%%%%%%%%%%%%%%%%%%%%%%%%%%%%%%%%%%%%%%%%%%%%%%%%%%%%%%%%%%%%%%%%%%%
\section{Relationship to $\fix$ and $\cyc$ (p,q)-Eulerian polynomials}\label{section04}
%%%%%%%%%%%%%%%%%%%%%%%%%%%%%%%%%%%%%%%%%%%
%%%%%%%%%%%%%%%%%%%%%%%%%%%%%%%%%%%%%%%%%%%
\subsection{A fundamental lemma on (p,q)-Eulerian polynomials}
%%%%%%%%%%%%%%%%%%%%%%%%%%%%
%\hspace*{\parindent}
%%\subsection{%%%%%%%%%%%%%%%%%%%%%%%%%%%%%%%%%%%%%%%%%%%%%%%%%%%
%%%%%%%%%%%%%%%%%%%%%%%%%%%%%%%%%%%%%%%%%%%%%%%%%%%%%%%%%%%
%%%%%%%%%%%%%%%%%%%%%%%%%%%5%%%%
%%%%%%%%%%%%%%%%%%%%%%%%%%%5%%%%
%%%%%%%%%%%%%%%%%%%%%%%%%%%%%%%%%%%%%%%%%%%
%%%%%%%%%%%%%%%%%%%%%%%%%%%%
\hspace*{\parindent}
%%\subsection{%%%%%%%%%%%%%%%%%%%%%%%%%%%%%%%%%%%%%%%%%%%%%%%%%%%
%%%%%%%%%%%%%%%%%%%%%%%%%%%%%%%%%%%%%%%%%%%%%%%%%%%%%%%%%%%
%%%%%%%%%%%%%%%%%%%%%%%%%%%5%%%%

The $\fix$ and $\cyc$ {\it $(p,q)$-Eulerian polynomials} $A_n(x,y,p,q)$ are defined by
\begin{equation*}
A_n(x,y,p,q)=\sum_{\pi\in\msn}x^{\exc(\pi)}y^{\drop(\pi)}p^{\fix(\pi)}q^{\cyc(\pi)}.
\end{equation*}
This $(p,q)$-Eulerian polynomial contains a great deal of information about permutations and colored permutations, see~\cite{Zeng02,Ma2401,Mongelli13} for details.
In particular, according to Theorem~\cite[Theorem~3.6]{Ma2401}, when $0\leqslant p\leqslant 1$ and $0\leqslant q\leqslant 1$,
the polynomials $A_n(x,1,p,q)$ are alternatingly increasing.
The following result will be used repeatedly in our discussion.
\begin{lemma}[{\cite[Lemma~3.12,~Theorem~3.4]{Ma2401}}]\label{G1G2}
If
\begin{equation}\label{G1}
G_1=\{I\rightarrow Ipq, p\rightarrow xy, x\rightarrow xy, y\rightarrow xy\},
\end{equation}
then we have
$$D_{G_1}^n(I)=I\sum_{\pi\in\msn}x^{\exc(\pi)}y^{\drop(\pi)}p^{\fix(\pi)}q^{\cyc(\pi)}.$$
Consider the change of variable $u=xy$ and $v=x+y$. Then $D_{G_1}(I)=Ipq,~D_{G_1}(p)=u$, $D_{G_1}(u)=uv,~D_{G_1}(v)=2u$.
Setting
\begin{equation}\label{G2}
G_2=\{I\rightarrow Ipq, p\rightarrow u, u\rightarrow uv, v\rightarrow 2u\},
\end{equation}
then we get
\begin{equation}\label{DG2}
D_{G_2}^n(I)=I\sum_{i=0}^np^i\sum_{j=0}^{\lrf{(n-i)/2}}\gamma_{n,i,j}(q)u^jv^{n-i-2j},
\end{equation}
where \begin{equation}\label{bnij-combin}
\gamma_{n,i,j}(q)=\sum_{\pi\in \ms_{n,i,j}}q^{\cyc(\pi)}
\end{equation}
and $\ms_{n,i,j}=\{\pi\in\msn: \cda(\pi)=0,~\fix(\pi)=i,~\exc(\pi)=j\}$.
\end{lemma}
%%%%%%%%%%%%%%%%%%%%%%%%%%%%%%%%%%%%%%%%%%%
%%%%%%%%%%%%%%%%%%%%%%%%%%%%%%%%%%%%%%%%%%%
\subsection{Four-variable polynomials}
%%%%%%%%%%%%%%%%%%%%%%%%%%%%
%\hspace*{\parindent}
%%\subsection{%%%%%%%%%%%%%%%%%%%%%%%%%%%%%%%%%%%%%%%%%%%%%%%%%%%
%%%%%%%%%%%%%%%%%%%%%%%%%%%%%%%%%%%%%%%%%%%%%%%%%%%%%%%%%%%
%%%%%%%%%%%%%%%%%%%%%%%%%%%5%%%%
%%%%%%%%%%%%%%%%%%%%%%%%%%%5%%%%
%%%%%%%%%%%%%%%%%%%%%%%%%%%%%%%%%%%%%%%%%%%
%%%%%%%%%%%%%%%%%%%%%%%%%%%%
\hspace*{\parindent}

We can now present the first result of this section.
\begin{theorem}\label{Anpq}
We have
$$\sum_{\pi\in\ms_{n+1}}x^{\basc(\pi)}y^{\des(\pi)-\plrm(\pi)}s^{\suc(\pi)}t^{\plrm(\pi)}=\sum_{\pi\in\msn}x^{\exc(\pi)}y^{\drop(\pi)}{\left(\frac{t+s}{2}\right)}^{\fix(\pi)}2^{\cyc(\pi)}.$$
When $t=y$, it reduces to
$$\sum_{\pi\in\ms_{n+1}}x^{\basc(\pi)}y^{\des(\pi)}s^{\suc(\pi)}=\sum_{\pi\in\msn}x^{\exc(\pi)}y^{\drop(\pi)}{\left(\frac{y+s}{2}\right)}^{\fix(\pi)}2^{\cyc(\pi)}.$$
\end{theorem}
\begin{proof}
Consider a change of the grammar $G$ given by Lemma~\ref{LMxyst-G}. Set $LM=I,t+s=pq$, where $p=\frac{t+s}{2},q=2$, then
we get the substitution rules defined by~\eqref{G1}.
By Lemma~\eqref{G1G2}, we immediately get the desired expression.
This completes the proof.
\end{proof}

Combining Theorem~\ref{Anpq} and Propositions~\ref{Pnx-bi-gamma} and~\ref{Pnxspiral}, we get the following.
\begin{corollary}
For any $n\geqslant 1$, the following two polynomials are alternatingly increasing and spiral, respectively:
$$\sum_{\pi\in\msn}x^{\exc(\pi)}2^{\cyc(\pi)-\fix(\pi)},~\sum_{\pi\in\msn}x^{\exc(\pi)}2^{\cyc(\pi)}.$$
\end{corollary}
%%%%%%%%%%%%%%%%%%%%%%%%%%%%%%%%%%%%%%%%%%
%%%%%%%%%%%%%%%%%%%%%%%%%%%%%%%%%%%%%%%%%%%
\subsection{Five-variable polynomials}
%%%%%%%%%%%%%%%%%%%%%%%%%%%%
%\hspace*{\parindent}
%%\subsection{%%%%%%%%%%%%%%%%%%%%%%%%%%%%%%%%%%%%%%%%%%%%%%%%%%%
%%%%%%%%%%%%%%%%%%%%%%%%%%%%%%%%%%%%%%%%%%%%%%%%%%%%%%%%%%%
%%%%%%%%%%%%%%%%%%%%%%%%%%%5%%%%
%%%%%%%%%%%%%%%%%%%%%%%%%%%5%%%%
%%%%%%%%%%%%%%%%%%%%%%%%%%%%%%%%%%%%%%%%%%%
%%%%%%%%%%%%%%%%%%%%%%%%%%%%
\hspace*{\parindent}

In this subsection,
we shall consider the joint distribution of the numbers of successions, exterior peaks, double ascents and double descents.
We need some more definitions. For $\pi\in\msn$, let $\pi(0)=\pi(n+1)=0$. Then for $i\in [n]$, any entry $\pi(i)$ can be
classified according to one of the four cases:
\begin{itemize}
  \item a peak if $\pi(i-1)<\pi(i)>\pi(i+1)$;
  \item a valley if $\pi(i-1)>\pi(i)>\pi(i+1)$;
  \item a double ascent if $\pi(i-1)<\pi(i)<\pi(i+1)$;
  \item a double descent if $\pi(i-1)>\pi(i)>\pi(i+1)$.
\end{itemize}
Let $\pk(\pi)$ (resp.~$\val(\pi)$,~$\dasc(\pi)$,~$\ddes(\pi)$) denote the number of peaks (resp.~valleys, double ascents,
double descents) in $\pi$. It is clear that $\pk(\pi)=\val(\pi)+1$.
In recent years, these statistics have been extensively studied by using various techniques, including continued fractions~\cite{Zeng16,Sokal22} and
noncommutative symmetric functions~\cite{Gessel20,Zhuang16}.

\begin{definition}
We say that a value $\pi(i)$ is a {\it simsun succession} of $\pi$ if $\pi(i)+1$ lies to the right of $\pi(i)$ and
all the values (if any) between $\pi(i)$ and $\pi(i)+1$ are greater than $\pi(i)+1$.
\end{definition}
Let $\ssuc(\pi)$ denote the number of simsun successions of $\pi$. Clearly, $\suc(\pi)\leqslant \ssuc(\pi)$.
\begin{example}
For $\pi\in\ms_3$, we have $$\ssuc(\textbf{1}\textbf{2}3)=2,~\ssuc(\textbf{1}32)=1,~\ssuc(213)=0,$$
$$\ssuc(\textbf{2}31)=\ssuc(3\textbf{1}2)=1,~\ssuc({32}1)=0.$$
\end{example}

Consider the generalized Eulerian polynomials
$$A_n(\alpha_1,\alpha_2,\alpha_3,\alpha_4,s)=\sum_{\pi\in\msn}\alpha_1^{\pk(\pi)}\alpha_2^{\val(\pi)}\alpha_3^{\dasc(\pi)}\alpha_4^{\ddes(\pi)}s^{\ssuc(\pi)}.$$
In particular, $A_1(\alpha_1,\alpha_2,\alpha_3,\alpha_4,s)=\alpha_1$ and $A_2(\alpha_1,\alpha_2,\alpha_3,\alpha_4,s)=\alpha_1(s\alpha_3+\alpha_4)$.
\begin{theorem}\label{thms}
Let be $\gamma_{n,i,j}(q)$ defined by~\eqref{bnij-combin}. Then
we have
$$A_{n+1}(\alpha_1,\alpha_2,\alpha_3,\alpha_4,s)=\sum_{i=0}^n\left(\frac{s\alpha_3+
\alpha_4}{s+1}\right)^i\sum_{j=0}^{\lrf{(n-i)/2}}\gamma_{n,i,j}(s+1)\alpha_1^{j+1}\alpha_2^j(\alpha_3+\alpha_4)^{n-i-2j}.$$
In particular, $$A_{n+1}(1,1,1,1,s)=(1+s)(2+s)\cdots (n+s)=\sum_{k=0}^n\stirling{n}{k}(s+1)^k,$$
where $\stirling{n}{k}$ is the unsigned Stirling numbers of the first kind, i.e., $\stirling{n}{k}=\#\{\pi\in\msn: \cyc(\pi)=k\}$.
\end{theorem}
\begin{proof}
We claim that if $G=\{\alpha_1\rightarrow \alpha_1\alpha_4,~\alpha_2\rightarrow \alpha_2\alpha_3,~\alpha_3\rightarrow \alpha_1\alpha_2,~\alpha_4\rightarrow \alpha_1\alpha_2,~M\rightarrow sM\alpha_3\}$,
then we have
\begin{equation}\label{DGMs}
D_G^n(M\alpha_1)=MA_{n+1}(\alpha_1,\alpha_2,\alpha_3,\alpha_4,s).
\end{equation}
Assume that permutations are prepended and appended by $0$.
We now give a grammatical labeling on permutations to generate the generalized Eulerian polynomials:
\begin{itemize}
  \item [$(i)$] If $\pi(i)=n$, we label it as $^{\alpha_1}n^{M}$;
  \item [$(ii)$] If $\pi(i)$ is a peak and $\pi(i)\neq n$, we label it as $^{\alpha_1}\pi(i)^{\alpha_2}$;
  \item [$(iii)$] If $\pi(i)$ is a double ascent, we put a superscript $\alpha_3$ just before $\pi(i)$;
  \item [$(iv)$] If $\pi(i)$ is a double descent, we put a superscript $\alpha_4$ right after $\pi(i)$;
  \item [$(v)$] If $\pi(i)$ is a simsun succession, we put a subscript $s$ right after $\pi(i)$.
\end{itemize}
With this labeling, the weight of $\pi$ is given as follows:
$$M\alpha_1^{\pk(\pi)}\alpha_2^{\val(\pi)}\alpha_3^{\dasc(\pi)}\alpha_4^{\ddes(\pi)}s^{\ssuc(\pi)}.$$
Note that $\ms_1=\{^{\alpha_1}1^{M}\}$ and $\ms_2=\{^{\alpha_3}1_s^{\alpha_1}2^{M},~^{\alpha_1}2^{M}1^{\alpha_4}\}$.
Then the sum of weights of the elements in $\ms_2$ is given by $D_G(M\alpha_1)$. We now present an example to illustrate the general case. Let $\pi=134265\in\ms_6$,
The grammatical labeling of $\pi$ is given as follows:
$$^{\alpha_3}1_s^{\alpha_3}3_s^{\alpha_1}4^{\alpha_2}2^{\alpha_1}6^M5^{\alpha_4}.$$
When we insert $7$ into $\pi$, the generated weighted permutations and their corresponding substitution rules are listed as follows:
\begin{align*}
&^{\alpha_1}7^{M}1_s^{\alpha_3}3_s^{\alpha_1}4^{\alpha_2}2^{\alpha_1}6^{\alpha_2}5^{\alpha_4}\leftrightarrow \alpha_3\rightarrow \alpha_1\alpha_2;\\
&^{\alpha_3}1_s^{\alpha_1}7^{M}3_s^{\alpha_1}4^{\alpha_2}2^{\alpha_1}6^{\alpha_2}5^{\alpha_4}\leftrightarrow \alpha_3\rightarrow \alpha_1\alpha_2;\\
&^{\alpha_3}1_s^{\alpha_3}3_s^{\alpha_1}7^{M}4^{\alpha_4}2^{\alpha_1}6^{\alpha_2}5^{\alpha_4}\leftrightarrow \alpha_1\rightarrow \alpha_1\alpha_4;\\
&^{\alpha_3}1_s^{\alpha_3}3_s^{\alpha_3}4^{\alpha_1}7^M2^{\alpha_1}6^{\alpha_2}5^{\alpha_4}\leftrightarrow \alpha_2\rightarrow \alpha_2\alpha_3;\\
&^{\alpha_3}1_s^{\alpha_3}3_s^{\alpha_1}4^{\alpha_2}2^{\alpha_1}7^M6^{\alpha_4}5^{\alpha_4}\leftrightarrow \alpha_1\rightarrow \alpha_1\alpha_4;\\
&^{\alpha_3}1_s^{\alpha_3}3_s^{\alpha_1}4^{\alpha_2}2^{\alpha_3}6_s^{\alpha_1}7^M5^{\alpha_4}\leftrightarrow M\rightarrow sM\alpha_3;\\
&^{\alpha_3}1_s^{\alpha_3}3_s^{\alpha_1}4^{\alpha_2}2^{\alpha_1}6^{\alpha_2}5^{\alpha_1}7^M\leftrightarrow \alpha_4\rightarrow \alpha_1\alpha_2.
\end{align*}
Each insertion of $7$ corresponds to one substitution rule in $G$.
Continuing in this way, we can eventually generated all the weighted elements in $\msn$. This completes the proof of~\eqref{DGMs}.

Note that $$D_G(M\alpha_1)=M\alpha_1(s\alpha_3+\alpha_4),~D_G(s\alpha_3+\alpha_4)=(1+s)\alpha_1\alpha_2,$$
$$D_G(\alpha_1\alpha_2)=\alpha_1\alpha_2(\alpha_3+\alpha_4),~D_G(\alpha_3+\alpha_4)=2\alpha_1\alpha_2.$$
We make a change of variables. Setting $a=M\alpha_1,~b=s\alpha_3+\alpha_4,~u=\alpha_1\alpha_2$ and $v=\alpha_3+\alpha_4$, we get the following grammar:
$$G'=\{a\rightarrow ab,~b\rightarrow (1+s)u,~u\rightarrow uv,~v\rightarrow 2u\}.$$
Consider a change of the grammar $G'$. Set $a=I, b=pq$, where $p=\frac{b}{1+s},~q=1+s$, then
we get the grammar $G_2$ defined by~\eqref{G2}.
Substituting $I=\alpha_1,~p=\frac{s\alpha_3+\alpha_4}{1+s},~q=1+s,~u=\alpha_1\alpha_2$ and $v=\alpha_3+\alpha_4$ into~\eqref{DG2},
we immediately get the desired expression. This completes the proof.
\end{proof}

\begin{corollary}
We have
$$A_{n+1}(\alpha_1,\alpha_2,\alpha_3,\alpha_4,0)=\sum_{i=0}^n
\alpha_4^i\sum_{j=0}^{\lrf{(n-i)/2}}\gamma_{n,i,j}(1)\alpha_1^{j+1}\alpha_2^j(\alpha_3+\alpha_4)^{n-i-2j},$$
where $\gamma_{n,i,j}(1)=\#\{\pi\in\msn: \cda(\pi)=0,~\fix(\pi)=i,~\exc(\pi)=j\}$.
\end{corollary}

Let
\begin{equation*}\label{eq:proof3}
\gamma=\gamma(x,p,q;z)=1+\sum_{n=1}^\infty\sum_{i=0}^np^i\sum_{j=0}^{\lrf{(n-i)/2}}\gamma_{n,i,j}(q)x^j\frac{z^n}{n!}.
\end{equation*}
According to~\cite[Eq~(13)]{Ma2401}, we have
\begin{equation*}\label{Csxz-explicit}
\gamma(x,p,q;z)=e^{z\left(p-\frac{1}{2}\right)q}\left(\frac{\sqrt{1-4x}}{\sqrt{1-4x}\cosh\left(\frac{z}{2}\sqrt{1-4x}\right)-\sinh\left(\frac{z}{2}\sqrt{1-4x}\right)}\right)^q.
\end{equation*}
Define
$$A(\alpha_1,\alpha_2,\alpha_3,\alpha_4,s;z)=\sum_{n=0}^\infty \frac{1}{\alpha_1}A_{n+1}(\alpha_1,\alpha_2,\alpha_3,\alpha_4,s)\frac{z^n}{n!}.$$
By Theorem~\ref{thms}, we get the following.
\begin{corollary}
We have
\begin{align*}
A(\alpha_1,\alpha_2,\alpha_3,\alpha_4,s;z)&=\gamma\left(\frac{\alpha_1\alpha_2}{(\alpha_3+\alpha_4)^2},\frac{s\alpha_3+
\alpha_4}{(s+1)(\alpha_3+\alpha_4)},1+s;(\alpha_3+\alpha_4)z\right)\\
&=1+(s\alpha_3+\alpha_4)z+(\alpha_1\alpha_2(1+s)+(s\alpha_3+\alpha_4)^2)\frac{z^2}{2!}+\cdots.
\end{align*}
\end{corollary}
%%%%%%%%%%%%%%%%%%%%%%%%%%%%%%%%%%%%%%%%%%%
%%%%%%%%%%%%%%%%%%%%%%%%%%%%%%%%%%%%%%%%%%%
\subsection{Six-variable polynomials}
%%%%%%%%%%%%%%%%%%%%%%%%%%%%
%\hspace*{\parindent}
%%\subsection{%%%%%%%%%%%%%%%%%%%%%%%%%%%%%%%%%%%%%%%%%%%%%%%%%%%
%%%%%%%%%%%%%%%%%%%%%%%%%%%%%%%%%%%%%%%%%%%%%%%%%%%%%%%%%%%
%%%%%%%%%%%%%%%%%%%%%%%%%%%5%%%%
%%%%%%%%%%%%%%%%%%%%%%%%%%%5%%%%
%%%%%%%%%%%%%%%%%%%%%%%%%%%%%%%%%%%%%%%%%%%
%%%%%%%%%%%%%%%%%%%%%%%%%%%%
\hspace*{\parindent}

Recall that $$A_n(x,y,p,q)=\sum_{\pi\in\msn}x^{\exc(\pi)}y^{\drop(\pi)}p^{\fix(\pi)}q^{\cyc(\pi)}.$$
Using the {\it exponential formula}, Ksavrelof-Zeng~\cite{Zeng02} found that
\begin{equation*}\label{Anxpq-EGF}
\sum_{n=0}^\infty A_n(x,1,p,q)\frac{z^n}{n!}=\left(\frac{(1-x)\mathrm{e}^{pz}}{\mathrm{e}^{xz}-x\mathrm{e}^{z}}\right)^q.
\end{equation*}
Since $\exc(\pi)+\drop(\pi)+\fix(\pi)=n$ for $\pi\in\msn$, it follows that
\begin{equation}\label{Anxpq-EGF}
\sum_{n=0}^\infty A_n(x,y,p,q)\frac{z^n}{n!}=\left(\frac{(y-x)\mathrm{e}^{pz}}{y\mathrm{e}^{xz}-x\mathrm{e}^{yz}}\right)^q.
\end{equation}

We now provide a generalization of Lemma~\ref{lemmaLMSt}, which can be proved in the same way.
\begin{lemma}\label{LMpq}
If $G=\{L\rightarrow pLt,M\rightarrow qMs,s\rightarrow xy, t\rightarrow xy, x\rightarrow xy,y\rightarrow xy\}$,
then we have
\begin{equation*}
D_{G}^n(LM)=LM\sum_{\pi\in\ms_{n+1}}x^{\basc(\pi)}y^{\des(\pi)-\plrm(\pi)}s^{\suc(\pi)}t^{\plrm(\pi)}p^{\lrm(\pi)-1}q^{\ssuc(\pi)}.
\end{equation*}
\end{lemma}

\begin{theorem}\label{thm28}
We have
\begin{equation*}\label{bascssuc}
\sum_{\pi\in\ms_{n+1}}x^{\basc(\pi)}y^{\des(\pi)-\plrm(\pi)}s^{\suc(\pi)}t^{\plrm(\pi)}p^{\lrm(\pi)-1}q^{\ssuc(\pi)}=A_n\left(x,y,\frac{pt+qs}{p+q},p+q\right),
\end{equation*}
which implies that $(\suc,\plrm)$ and $(\lrm(\pi)-1,\ssuc)$ are both symmetric distribution.
In particular,
\begin{equation*}
\sum_{\pi\in\ms_{n+1}}s^{\suc(\pi)}t^{\plrm(\pi)}p^{\lrm(\pi)-1}q^{\ssuc(\pi)}=\sum_{\pi\in\msn}{\left(\frac{pt+qs}{p+q}\right)}^{\fix(\pi)}(p+q)^{\cyc(\pi)},
\end{equation*}
\begin{equation*}\label{excfix}
\sum_{\substack{\pi\in\ms_{n+1}\\\ssuc(\pi)=0}}x^{\basc(\pi)}t^{\plrm(\pi)}=\sum_{\pi\in\msn}x^{\exc(\pi)}t^{\fix(\pi)},
\end{equation*}
\begin{equation*}\label{bascsimsun}
\sum_{\pi\in\ms_{n+1}}x^{\basc(\pi)}p^{\lrm(\pi)-1}q^{\ssuc(\pi)}=\sum_{\pi\in\msn}x^{\exc(\pi)}(p+q)^{\cyc(\pi)}.
\end{equation*}
\end{theorem}
\begin{proof}
Let $G$ be the grammar given by Lemma~\ref{LMpq}. Note that $D_G(LM)=LM(pt+qs)$ and $D_G(pt+qs)=(p+q)xy$.
Letting $LM\rightarrow I$, $\frac{pt+qs}{p+q}\rightarrow p$ and $p+q\rightarrow q$,
we obtain the substitution rules defined by~\eqref{G1}.
By Lemma~\eqref{G1G2} and~\eqref{Anxpq-EGF}, we immediately get the desired expression.
This completes the proof.
\end{proof}

By~\cite[Theorem~3.6]{Ma2401}, we see that if $t\in [0,1]$ and $q\in[-1,0]$, then the following two polynomials are alternatingly increasing:
$$\sum_{\substack{\pi\in\ms_{n}\\\ssuc(\pi)=0}}x^{\basc(\pi)}t^{\plrm(\pi)},~\sum_{\pi\in\ms_{n+1}}x^{\basc(\pi)}q^{\ssuc(\pi)}.$$
Combining~\eqref{Anxpq-EGF} and Theorem~\ref{thm28}, we can give the following generalization of~\eqref{A-EGF}.
\begin{corollary}\label{Corend}
We have
\begin{align*}
&\sum_{n=0}^\infty\sum_{\pi\in\ms_{n+1}}x^{\basc(\pi)}y^{\des(\pi)-\plrm(\pi)}s^{\suc(\pi)}t^{\plrm(\pi)}p^{\lrm(\pi)-1}q^{\ssuc(\pi)}\frac{z^n}{n!}\\
&=\mathrm{e}^{(pt+qs)z}\left(\frac{y-x}{y\mathrm{e}^{xz}-x\mathrm{e}^{yz}}\right)^{p+q}\\
&=1+(qs+pt)z+\left((qs+pt)^2+(p+q)xy\right)\frac{z^2}{2}+\\
&\left((qs+pt)^3+3(p+q)(qs+pt)xy+(p+q)xy(x+y)\right)\frac{z^3}{3!}+\cdots.
\end{align*}
\end{corollary}
%%%%%%%%%%%%%%%%%%%%%%%%%%%%%%%%%%%
%%%%%%%%%%%%%%%%%%%%%%%%%%%%%%%%%%%%%%%%%%%%%%%%%%%%%%%%%%%%%%%%%%%%%%%%%%%%%
%%%%%%%%%%%%%%%%%%%%%%%%%%%%%%%%%%%%%%%%%%%%%%%%%%%%%%%%%%%%%%%%%%%%%%%%%%%%%
\bibliographystyle{amsplain}

\end{document}